\documentclass{amsproc}
\numberwithin{equation}{section}
\newtheorem{theorem}{Theorem}[section]
\newtheorem{lemma}[theorem]{Lemma}

\begin{document}

\title[Multiple orthogonal polynomials]{Ratio asymptotics for multiple orthogonal polynomials}
\author{Walter Van Assche}
\address{KU Leuven, Department of Mathematics, Celestijnenlaan 200 B box 2400, BE-3001 Leuven, Belgium}
\email{walter@wis.kuleuven.be}
\thanks{Supported by KU Leuven research grant OT/12/073, FWO research project G.0934.13, and the Belgian Interuniversity Attraction Poles programme P7/18.}

\date{\today}

\begin{abstract}
We give the asymptotic behavior of the ratio of two neighboring multiple orthogonal polynomials
under the condition that the recurrence coefficients in the nearest neighbor recurrence relations
converge. 
\end{abstract}

\maketitle

\section{Introduction}

Multiple orthogonal polynomials (or Hermite-Pad\'e polynomials) are monic polynomials $P_{\vec{n}}$ with a multi-index 
$n=(n_1,\ldots,n_r) \in \mathbb{N}^r$ and of degree $|\vec{n}|=n_1+\cdots +n_r$ satisfying the orthogonality relations
\[  \int P_{\vec{n}}(x) x^k \, d\mu_j(x) = 0, \qquad k=0,1,\ldots,n_j-1, \ 1 \leq j \leq r, \]
for a system of $r$ positive measures $(\mu_1,\ldots,\mu_r)$ on the real line. They  
satisfy a system of nearest neighbor recurrence relations \cite[\S 23.1.4]{Ismail} \cite{WVA2}
\begin{equation}  \label{rrel}
  x P_{\vec{n}}(x) = P_{\vec{n}+\vec{e}_k}(x) + b_{\vec{n},k} P_{\vec{n}}(x) +
    \sum_{j=1}^r a_{\vec{n},j} P_{\vec{n}-\vec{e}_j}(x),   
\end{equation}
for $1 \leq k \leq r$. We will assume that the recurrence coefficients satisfy the following. Let
$n_j = \lfloor q_j n \rfloor$, where $q_j > 0$ and $\sum_{j=1}^r q_j = 1$, so that $|\vec{n}|/ n \to 1$
as $n \to \infty$. We say that multiple orthogonal polynomials belong to the class $M(\vec{a},\vec{b})$ if
\begin{equation} \label{limab}
    \lim_{n \to \infty} a_{\vec{n},j} = a_j, \quad \lim_{n \to \infty} b_{\vec{n},j} = b_j.  
\end{equation}
Observe that the limits $\vec{a} = (a_1,\ldots,a_r)$ and $\vec{b} = (b_1,\ldots,b_r)$ depend on the parameters
$(q_1,\ldots,q_r)$ which indicate the direction in $\mathbb{N}^r$ in which the multi-index $\vec{n}$ tends to infinity. Our interest is to obtain the asymptotic behavior of the ratio $P_{\vec{n}+\vec{e}_k}(x)/P_{\vec{n}}(x)$ of two neighboring multiple orthogonal polynomials. The main result is

\begin{theorem} \label{thm1}
Suppose the multiple orthogonal polynomials $P_{\vec{n}}$ and $P_{\vec{n}+\vec{e}_k}$ have interlacing real zeros for every $\vec{n}$ 
and $1 \leq k \leq r$ and the recurrence coefficients have asymptotic behavior given by
\eqref{limab}, where $n_j = \lfloor q_j n \rfloor$, with $q_j > 0$ and $\sum_{j=1}^r q_j = 1$, and $b_i \neq b_j$ whenever $i \neq j$. Then
\[  \lim_{n \to \infty} \frac{P_{\vec{n}+\vec{e}_k}(x)}{P_{\vec{n}}(x)} = z(x)- b_k   \]
uniformly on compact subsets $K$ of $\mathbb{C} \setminus \mathbb{R}$, where
$z$ is the solution of the algebraic equation
\begin{equation}  \label{algeq}
     (z-x) B_r(z) + A_{r-1}(z) = 0  
\end{equation}
for which $z(x) - x \to 0$ when $x \to \infty$,
where $B_r(z) = (z-b_1)(z-b_2)\cdots (z-b_r)$ and $A_{r-1}$ is the polynomial of degree $r-1$ for which
\begin{equation}   \label{parfrac1}
 \frac{A_{r-1}(z)}{B_r(z)} = \sum_{j=1}^r \frac{a_j}{z-b_j} . 
\end{equation}
\end{theorem}  

The condition that $P_{\vec{n}}$ and $P_{\vec{n}+\vec{e}_k}$ have interlacing real zeros for every $\vec{n}$ and $1 \leq k \leq r$ is not easy to
check, but there are sufficient conditions that give this interlacing property. 
For instance, if $a_{\vec{n},j} >0$ for every $\vec{n}$ for which $n_j >0$, then
the interlacing property holds \cite[Thm.~2.2]{HVA}.

If the recurrence coefficients are unbounded, then one can investigate the ratio of two neighboring multiple orthogonal polynomials in which the variable
is scaled, taking into account the growth of the recurrence coefficients. 

\begin{theorem} \label{thm2}
Suppose the multiple orthogonal polynomials $P_{\vec{n}}$ and $P_{\vec{n}+\vec{e}_k}$ have interlacing real zeros for every $\vec{n}$ and $1 \leq k \leq r$
and that for some $\gamma >0$
\begin{equation} \label{limabN}
    \lim_{n \to \infty} \frac{a_{\vec{n},j}}{n^{2\gamma}} = a_j, \quad \lim_{n \to \infty} \frac{b_{\vec{n},j}}{n^\gamma} = b_j,  
\end{equation}
where $\vec{n} = (\lfloor q_1 n \rfloor,\ldots,\lfloor q_r n\rfloor)$, with $q_j > 0$ and $\sum_{j=1}^r q_j = 1$, and $b_i \neq b_j$ whenever $i \neq j$. Then
\[  \lim_{n \to \infty} \frac{P_{\vec{n}+\vec{e}_k}(n^\gamma x)}{n^\gamma P_{\vec{n}}(n^\gamma x)} = z(x)- b_k   \]
uniformly on compact subsets $K$ of $\mathbb{C} \setminus \mathbb{R}$, where
$z$ is the solution of the algebraic equation
\begin{equation}  \label{algeq2}
     (z-x) B_r(z) + A_{r-1}(z) = 0  
\end{equation}
for which $z(x) - x \to 0$ when $x \to \infty$,
where $B_r(z) = (z-b_1)(z-b_2)\cdots (z-b_r)$ and $A_{r-1}$ is the polynomial of degree $r-1$ for which
\begin{equation} \label{parfrac2}
   \frac{A_{r-1}(z)}{B_r(z)} = \sum_{j=1}^r \frac{a_j}{z-b_j} . 
\end{equation}
\end{theorem}  

The multiple orthogonal polynomials $(Q_m)_{m \in \mathbb{N}}$ on the stepline, i.e., for $m=kr+j$ we set 
$\vec{n}_m=(k+1,\ldots,k+1,k,\ldots,k)$, with $j$ times $k+1$ and $r-j$ times $k$,
and $Q_m = P_{\vec{n}_m}$, satisfy a higher order recurrence relation of the form
\[   xQ_m(x) = Q_{m+1}(x) + \sum_{k=0}^{r} \beta_{m,k} Q_{m-k}(x), \]
and many authors investigated the multiple orthogonal polynomials on the stepline and the recurrence coefficients $\beta_{m,k}$.
Kalyagin \cite{Kalyagin} showed that for Angelesco systems (i.e., the measures $\mu_j$ are supported on pairwise disjoint intervals $\Delta_j$)
the recurrence coefficients $(\beta_{m,k})_m$ are asymptotically periodic for $0 \leq k \leq r$ and he 
investigated the asymptotic behavior of the polynomials 
(assuming a Szeg\H{o} condition for each $\mu_j$ on $\Delta_j$) and some ratio asymptotics. Aptekarev, Kalyagin and Saff \cite{ApKaSaff} 
investigated the ratio asymptotics of the stepline polynomials when $\beta_{m,k}=0$ for $0 \leq k \leq r-1$ and $\lim_{m \to \infty} \beta_{m,r} = b$.
Aptekarev et al. \cite{ApKaLoRo} showed that the recurrence coefficients $(\beta_{m,k})_m$ are asymptotically periodic for Nikishin systems, and
the ratio of neighboring multiple orthogonal polynomials for Nikishin systems was investigated in \cite{ApLoRo}, \cite{LoLo} and \cite{DeLoLo}.
In this paper we are interested in the ratio asymptotics of multiple orthogonal polynomials when the multi-index tends to infinity in any
direction in $\mathbb{N}^r$. The recurrence relation for the stepline polynomials only gives asymptotics for the multiple orthogonal polynomials
near the diagonal, i.e., the case when $q_j=1/r$ for $1 \leq j \leq r$. The nearest neighbor recurrence relations allow us to investigate
the asymptotic behavior in any direction. Our Theorems \ref{thm1} and \ref{thm2} show that, under the conditions that we impose, 
the ratio asymptotics for two neighboring multiple orthogonal polynomials is given in terms of an algebraic function $z(x)$ which is the solution of an algebraic equation \eqref{algeq} of order $r+1$. The Riemann surface for this algebraic function has genus 0 since for given $z$ one can find 
$x$ in a unique way from \eqref{algeq}. We illustrate our results using some known systems of multiple orthogonal polynomials in Section \ref{sec:ex}.

\section{Proof of Theorem \ref{thm1}}

Use the recurrence relation \eqref{rrel} and divide by $P_{\vec{n}}(x)$ to find
\begin{equation}  \label{ratiorec}
   x = \frac{P_{\vec{n}+\vec{e}_k}(x)}{P_{\vec{n}}(x)} + b_{\vec{n},k} + \sum_{j=1}^r a_{\vec{n},j} \frac{P_{\vec{n}-\vec{e}_j}(x)}{P_{\vec{n}}(x)}.
\end{equation}
The interlacing of the zeros of $P_{\vec{n}}$ and $P_{\vec{n}-\vec{e}_j}$ implies the partial fractions decomposition
\[   \frac{P_{\vec{n}-\vec{e}_j}(x)}{P_{\vec{n}}(x)} = \sum_{k=1}^{|\vec{n}|} \frac{A_{\vec{n},k}}{x-x_{\vec{n},k}}, \]
with positive residues $A_{\vec{n},k} > 0$, where $\{x_{\vec{n},k}, \ 1 \leq k \leq |\vec{n}| \}$ are the zeros of $P_{\vec{n}}$.
Furthermore, since all our multiple orthogonal polynomials are monic, we have
\[   \sum_{k=1}^{|\vec{n}|} A_{\vec{n},k} = 1.  \]
Hence for $x \in K \subset \mathbb{C} \setminus \mathbb{R}$ we have the bound
\begin{equation}  \label{Pbound}
  \left|  \frac{P_{\vec{n}-\vec{e}_j}(x)}{P_{\vec{n}}(x)}  \right| \leq \sum_{k=1}^{|\vec{n}|} \frac{A_{\vec{n},k}}{|x-x_{\vec{n},k}|} \leq \frac{1}{\delta}, 
\end{equation}
where $\delta = \min \{|x-y| : \ x \in K,\ y \in \mathbb{R} \} > 0$ is the minimal distance between $K$ and $\mathbb{R}$.
Hence $\{ P_{\vec{n}-\vec{e}_j}(x)/P_{\vec{n}}(x) : \ \vec{n} \in \mathbb{N}^r \}$ is a normal family on every compact $K \subset \mathbb{C} \setminus \mathbb{R}$, and by Montel's theorem there exists a subsequence $(n_k)_{k \in \mathbb{N}}$ so that
$P_{\vec{n}_k-\vec{e}_j}(x)/P_{\vec{n}_k}(x)$ converges uniformly on $K$, where $\vec{n}_k = (\lfloor q_1n_k\rfloor,\ldots,\lfloor q_r n_k \rfloor)$.
By taking a further subsequence we can conclude that
there is a subsequence $(n_k)_{k \in \mathbb{N}}$ such that
\[    \lim_{k \to \infty}  \frac{P_{\vec{n}_k-\vec{e}_j}(x)}{P_{\vec{n}_k}(x)} = h_j(x)  \]
uniformly on $K$ for every $j$ with $1 \leq j \leq r$. Clearly $h_j$ is analytic on $K$ and $h_j(x) \to 0$ when $x \to \infty$.
From Lemma \ref{lem} in Section \ref{sec:lem} we find that along this subsequence we also have
\[    \lim_{k \to \infty} \frac{P_{\vec{n}_k+\vec{e}_j}(x)}{P_{\vec{n}_k}(x)} = \frac{1}{h_j(x)} ,  \]
uniformly on $K$. Hence if we take the limit along this subsequence in \eqref{ratiorec} and use \eqref{limab}, then
\[    x =  \frac{1}{h_k(x)} + b_k + \sum_{j=1}^r a_j h_j(x)  .  \]
Define $z=z(x)$ by
\begin{equation}  \label{defz}    
z = x - \sum_{j=1}^r a_j h_j(x), 
\end{equation}
then we find that
\[    h_k(x) = \frac{1}{z-b_k}   , \qquad  1 \leq k \leq r, \]
so that \eqref{defz} implies
\[          x-z  =  \sum_{j=1}^r \frac{a_j}{z-b_j} = \frac{A_{r-1}(z)}{B_r(z)},  \]
which gives the algebraic equation \eqref{algeq}. For $x \to \infty$ we either have that $z(x)$ remains bounded or $z(x)$ is unbounded.
In the latter case
\[  \lim_{x \to \infty} \frac{A_{r-1}(z)}{B_r(z)} = 0,  \]
so that $z(x) - x \to 0$, which is the solution that we want. If $z(x)$ remains bounded, then  $A_{r-1}(z)/B_r(z)$ is unbounded as $x \to \infty$,
which is only possible when  $z(x) \to b_j$. This gives $r$ other solutions of the algebraic equation \eqref{algeq}.
Hence along the subsequence $(n_k)_{k \in \mathbb{N}}$ we found the limits $1/h_j(x) = z(x)-b_j$ for $1 \leq j \leq r$. 
The solution $z$ is independent of the subsequence $(n_k)_{k \in \mathbb{N}}$, 
hence every converging subsequence gives the same limit functions, which implies that the full sequence converges.

\section{Some technical lemmas}  \label{sec:lem}

An important step in the proof is that $P_{\vec{n}-\vec{e}_j}(x)/P_{\vec{n}}(x)$ and $P_{\vec{n}}(x)/P_{\vec{n}+\vec{e}_j}(x)$
have the same limit as $n \to \infty$. To prove this, one needs some extra results.

\begin{lemma} \label{estimate}
Suppose $\{D_{\vec{n}}, \vec{n} \in \mathbb{N}^r\}$ are positive quantities, with $D_{\vec{n}} = 0$ whenever $n_j < 0$ for some $j \in \{1,\ldots, r\}$, satisfying
\[    D_{\vec{n}} \leq  A_{\vec{n}} + \sum_{j=1}^r a_j D_{\vec{n}-\vec{e}_j}, \]
with given quantities $A_{\vec{n}} > 0$ and $a_j > 0$ $(1 \leq j \leq r)$.
Then
\[     D_{\vec{n}} \leq \sum_{\vec{m} \leq \vec{n}} A_{\vec{n}-\vec{m}} \binom{|\vec{m}|}{m_1,\ldots,m_r} a_1^{m_1} \cdots a_r^{m_r}. \]
\end{lemma}
 
A special case is when $A_{\vec{n}}=A$ for every $\vec{n}$, in which case one has
\[    D_{\vec{n}} \leq A \sum_{k=0}^{|\vec{n}|} \left( \sum_{j=1}^r a_j \right)^k .  \]  

\begin{proof}
We will use induction on the length $|\vec{n}|$.
For $|\vec{n}|=0$ we have $\vec{n} = \vec{0}$. Since $D_{\vec{n}} = 0$ whenever $n_j < 0$ for some $j \in \{1,\ldots, r\}$ we find that
$D_{\vec{0}} \leq A_{\vec{0}}$, which corresponds to the required result when $\vec{n}=\vec{0}$.

Suppose that the result is true for every multi-index of length $|\vec{n}|-1$. Then
\[  D_{\vec{n}} \leq A_{\vec{n}} + \sum_{j=1}^r a_j \sum_{\vec{m} \leq \vec{n}-\vec{e}_j} A_{\vec{n}-\vec{e}_j-\vec{m}}
     \binom{|\vec{m}|}{m_1,\ldots,m_r} a_1^{m_1} \cdots a_r^{m_r}.  \]
In the second sum we change the index $m_j$ to $m_j^*-1$ to find
\[  \sum_{\vec{m}^* \leq \vec{n}, m_j^*\neq 0} A_{\vec{n}-\vec{m}^*}
     \binom{|\vec{m}^*|-1}{m_1,\ldots,m_j^*-1,\ldots,m_r} a_1^{m_1} \cdots a_j^{m_j^*-1}\cdots a_r^{m_r} , \]
where $\vec{m}^* = (m_1,\ldots,m_j^*,\ldots,m_r)$.
Now use for $|\vec{m}^*| \neq 0$ the identity
\[    \binom{|\vec{m}^*|-1}{m_1,\ldots,m_j^*-1,\ldots,m_r} = \frac{m_j^*}{|\vec{m}^*|} 
      \binom{|\vec{m}^*|}{m_1,\ldots,m_j^*,\ldots,m_r} \]
to find that the sum becomes
\[   \frac{1}{a_j|\vec{m}^*|} \sum_{\vec{m}^* \leq \vec{n}, m_j^*\neq 0} m_j^* A_{\vec{n}-\vec{m}^*}
     \binom{|\vec{m}^*|}{m_1,\ldots,m_j^*,\ldots,m_r} a_1^{m_1} \cdots a_j^{m_j^*}\cdots a_r^{m_r} . \]
This gives
\[   D_{\vec{n}} \leq A_{\vec{n}} + 
    \sum_{\vec{0} \neq \vec{m}^* \leq \vec{n}} \left( \sum_{j=1}^r \frac{m_j^*}{|\vec{m}^*|} \right)  A_{\vec{n}-\vec{m}^*}
     \binom{|\vec{m}^*|}{m_1,\ldots,m_r} a_1^{m_1} \cdots a_r^{m_r}, \]
which immediately gives the required result.
\end{proof}

\begin{lemma}  \label{lim}
Suppose that $n_j = \lfloor q_j n \rfloor$ and $\sum_{j=1}^r q_j = 1$ and that
\[  \lim_{n \to \infty} A_{\vec{n}} = 0. \]
Then, if $0< a_j < a$, where $0 < a < 1/r$ is such that $a^{q_i} < 1/(1+(r-1)a)$ for $1 \leq i \leq r$, one has
\[  \lim_{n \to \infty} \sum_{\vec{m} \leq \vec{n}} A_{\vec{n}-\vec{m}} \binom{|\vec{m}|}{m_1,\ldots,m_r} a_1^{m_1} \cdots a_r^{m_r} = 0. \]  
\end{lemma}
\begin{proof}
For every $\varepsilon >0$ there exists $n_0 \in  \mathbb{N}$ such that for all $n \geq n_0$ we have
that $|A_{\vec{n}}| < \varepsilon$ whenever $n_j \geq \lfloor n_0q_j \rfloor$ for every $j \in \{1,\ldots,r\}$.
We estimate the sum by
\[  \sum_{\vec{m} \leq \vec{n}} |A_{\vec{n}-\vec{m}}| \binom{|\vec{m}|}{m_1,\ldots,m_r} a_1^{m_1} \cdots a_r^{m_r}
  \leq \sum_{\vec{m} \in I_0} \cdots  + \sum_{j=1}^r \sum_{\vec{m} \in I_j} \cdots  \]
where $I_0 = \{\vec{m} \in \mathbb{N}^r: \ \vec{m} \leq \vec{n}-\vec{n}_0\}$ with $\vec{n}_0 = (\lfloor n_0q_1 \rfloor,\ldots,\lfloor n_0q_r \rfloor)$, 
and $I_j = \{ \vec{m} \in \mathbb{N}^r : \ n_j - \lfloor n_0q_j \rfloor < m_j \leq n_j \}$. Note that the sets $I_1,\ldots,I_r$ are not disjoint
so that we indeed get an upper bound. We have
\begin{eqnarray*}    \sum_{\vec{m} \in I_0} |A_{\vec{n}-\vec{m}}| \binom{|\vec{m}|}{m_1,\ldots,m_r} a_1^{m_1} \cdots a_r^{m_r}
    &\leq& \varepsilon \sum_{\vec{m} \in I_0} \binom{|\vec{m}|}{m_1,\ldots,m_r} a_1^{m_1} \cdots a_r^{m_r}  \\
    &\leq& \varepsilon  \sum_{\vec{m} \leq \vec{n}} \binom{|\vec{m}|}{m_1,\ldots,m_r} a_1^{m_1} \cdots a_r^{m_r} \\
    & = &  \varepsilon  \sum_{k=0}^{|\vec{n}|}  \left( \sum_{j=1}^r a_j \right)^k   \\
    & \leq & \frac{\varepsilon}{1-ra}, 
\end{eqnarray*}
where we used that $a_j < a < 1/r$ for $1 \leq j \leq r$. 
On the other hand we know that the $A_{\vec{n}}$ are bounded, so that $|A_{\vec{n}}| \leq M$ for some $M >0$, and then
\begin{multline*}   \sum_{\vec{m} \in I_j} |A_{\vec{n}-\vec{m}}| \binom{|\vec{m}|}{m_1,\ldots,m_r} a_1^{m_1} \cdots a_r^{m_r}  \\
    \leq M a_j^{n_j-\lfloor n_0 q_j \rfloor} \sum_{\vec{m} \in I_j} \binom{|\vec{m}|}{m_1,\ldots,m_r} a_1^{m_1} \cdots 1 \cdots a_r^{m_r} ,
\end{multline*}
where we have used $a_j^{m_j} \leq a_j^{n_j-\lfloor n_0 q_j \rfloor}$, which holds on $I_j$. Furthermore
we have
\begin{eqnarray*}
  M a_j^{n_j-\lfloor n_0 q_j \rfloor} \lefteqn{\sum_{\vec{m} \in I_j} \binom{|\vec{m}|}{m_1,\ldots,m_r} a_1^{m_1} \cdots 1 \cdots a_r^{m_r}} & & \\
    &\leq& M a_j^{n_j-\lfloor n_0 q_j \rfloor} \sum_{\vec{m} \leq \vec{n}} \binom{|\vec{m}|}{m_1,\ldots,m_r} a_1^{m_1} \cdots 1 \cdots a_r^{m_r} \\
     &=&  M a_j^{n_j-\lfloor n_0 q_j \rfloor} \sum_{k=0}^{|\vec{n}|} \left( 1 + \sum_{i=1,i\neq j}^r a_i \right)^k \\
     &\leq& M  a^{n_j-\lfloor n_0 q_j \rfloor} \frac{(1+(r-1)a)^{|\vec{n}|+1}-1}{(r-1)a},
\end{eqnarray*}
where we used that $a_j \leq a$. Clearly, 
\[    \lim_{n \to \infty} a^{n q_j} (1+(r-1)a)^n = 0 \]
whenever $a^{q_j} (1+(r-1)a) < 1$, and this is indeed what was assumed to be true for $a$. Combining our estimates we have
\[  \limsup_{n \to \infty} \left|   \sum_{\vec{m} \leq \vec{n}} A_{\vec{n}-\vec{m}} \binom{|\vec{m}|}{m_1,\ldots,m_r} a_1^{m_1} \cdots a_r^{m_r} \right|
   \leq \frac{\varepsilon}{1-ra} , \]
and since this holds for every $\varepsilon > 0$, the required result follows.
\end{proof}

\begin{lemma} \label{lem}
Suppose that \eqref{limab} holds, where $n_j = \lfloor q_j n\rfloor$, with $q_j > 0$ and $\sum_{j=1}^r q_j = 1$. Then
we have, uniformly on compact subsets $K \subset \mathbb{C} \setminus \mathbb{R}$,
\[   \lim_{n \to \infty} \left| \frac{P_{\vec{n}}(x)}{P_{\vec{n}+\vec{e}_k}(x)} -
     \frac{P_{\vec{n}-\vec{e}_\ell}(x)}{P_{\vec{n}+\vec{e}_k-\vec{e}_\ell}(x)} \right| = 0, \]
for any $k$ and $\ell$ in $\{1,\ldots,r\}$. 
\end{lemma}
\begin{proof}
From the recurrence relation \eqref{rrel} we have
\[   x  = \frac{P_{\vec{n}+\vec{e}_k}(x)}{P_{\vec{n}}(x)} + b_{\vec{n},k} 
 + \sum_{j=1}^{r} a_{\vec{n},j} \frac{P_{\vec{n}-\vec{e}_j}(x)}{P_{\vec{n}}(x)}. \] 
The same relation but with $\vec{n}$ replaced by $\vec{n}-\vec{e}_\ell$ gives
\[   x  = \frac{P_{\vec{n}+\vec{e}_k-\vec{e}_\ell}(x)}{P_{\vec{n}-\vec{e}_\ell}(x)} + b_{\vec{n}-\vec{e}_\ell,k} 
 + \sum_{j=1}^{r} a_{\vec{n}-\vec{e}_\ell,j} \frac{P_{\vec{n}-\vec{e}_j-\vec{e}_\ell}(x)}{P_{\vec{n}-\vec{e}_\ell}(x)}. \]
Subtract both equations to find
\begin{eqnarray*}
     \frac{P_{\vec{n}+\vec{e}_k}(x)}{P_{\vec{n}}(x)} -  \frac{P_{\vec{n}+\vec{e}_k-\vec{e}_\ell}(x)}{P_{\vec{n}-\vec{e}_\ell}(x)}
  & = & b_{\vec{n}-\vec{e}_\ell,k} -  b_{\vec{n},k} - \sum_{j=1}^r a_{\vec{n},j} \left( \frac{P_{\vec{n}-\vec{e}_j}(x)}{P_{\vec{n}}(x)}
      - \frac{P_{\vec{n}-\vec{e}_j-\vec{e}_\ell}(x)}{P_{\vec{n}-\vec{e}_\ell}(x)} \right) \\
   & &  -\ \sum_{j=1}^r  (a_{\vec{n},j} - a_{\vec{n}-\vec{e}_\ell,j}) \frac{P_{\vec{n}-\vec{e}_j-\vec{e}_\ell}(x)}{P_{\vec{n}-\vec{e}_\ell}(x)} .
\end{eqnarray*}
If we use the bound \eqref{Pbound}, then
\[     \left| \frac{P_{\vec{n}+\vec{e}_k}(x)}{P_{\vec{n}}(x)} - \frac{P_{\vec{n}+\vec{e}_k-\vec{e}_\ell}(x)}{P_{\vec{n}-\vec{e}_\ell}(x)} \right| 
    \geq \delta^2 \left| \frac{P_{\vec{n}}(x)}{P_{\vec{n}+\vec{e}_k}(x)} - \frac{P_{\vec{n}-\vec{e}_\ell}(x)}{P_{\vec{n}+\vec{e}_k-\vec{e}_\ell}(x)} \right|  \]
so that
\begin{eqnarray*}
   \left| \frac{P_{\vec{n}}(x)}{P_{\vec{n}+\vec{e}_k}(x)} - \frac{P_{\vec{n}-\vec{e}_\ell}(x)}{P_{\vec{n}+\vec{e}_k-\vec{e}_\ell}(x)} \right|
   &\leq &  \frac{1}{\delta^2} \sum_{j=1}^r a_{\vec{n},j} \left| \frac{P_{\vec{n}-\vec{e}_j}(x)}{P_{\vec{n}}(x)}
      - \frac{P_{\vec{n}-\vec{e}_j-\vec{e}_\ell}(x)}{P_{\vec{n}-\vec{e}_\ell}(x)} \right| \\
   & &  +\ \frac{1}{\delta^2} |b_{\vec{n},k} -  b_{\vec{n}-\vec{e}_\ell,k}| 
   + \frac{1}{\delta^3} \sum_{j=1}^r  |a_{\vec{n},j} - a_{\vec{n}-\vec{e}_\ell,j}|.
\end{eqnarray*}
If we  use the notation
\[     D_{\vec{n},k,\ell} = \left| \frac{P_{\vec{n}}(x)}{P_{\vec{n}+\vec{e}_k}(x)} 
- \frac{P_{\vec{n}-\vec{e}_\ell}(x)}{P_{\vec{n}+\vec{e}_k-\vec{e}_\ell}(x)} \right|, \]
then this gives
\[   D_{\vec{n},k,\ell} \leq \frac{1}{\delta^2} |b_{\vec{n},k} -  b_{\vec{n}-\vec{e}_\ell,k}| + 
  \frac{1}{\delta^3} \sum_{j=1}^r  |a_{\vec{n},j} - a_{\vec{n}-\vec{e}_\ell,j}|
  + \frac{1}{\delta^2} \sum_{j=1}^r a_{\vec{n},j} D_{\vec{n}-\vec{e}_j,j,\ell}.  \]
The convergence \eqref{limab} implies that $a_{\vec{n},j} \leq \hat{a}_j$ for certain constants $\hat{a}_1,\ldots,\hat{a}_r$.  Now denote 
\[   D_{\vec{n},\ell} = \max_{1 \leq k \leq r} D_{\vec{n},k,\ell}, \]
then we arrive at the inequality
\[   D_{\vec{n},\ell} \leq \sum_{j=1}^r \left( \frac{|b_{\vec{n},j} -  b_{\vec{n}-\vec{e}_\ell,j}|}{\delta^2} + 
     \frac{|a_{\vec{n},j} - a_{\vec{n}-\vec{e}_\ell,j}|}{\delta^3} \right) + \sum_{j=1}^r \frac{\hat{a}_j}{\delta^2} D_{\vec{n}-\vec{e}_j,\ell}. \]
From Lemma \ref{estimate} we then find
\[    D_{\vec{n},\ell} \leq \sum_{\vec{m}\leq \vec{n}} A_{\vec{n}-\vec{m}} \binom{|\vec{m}|}{m_1,\ldots,m_r} 
  \frac{\hat{a}_1^{m_1} \cdots \hat{a}_r^{m_r}}{\delta^{2|\vec{m}|}} , \]
where
\[   A_{\vec{n}} = \sum_{j=1}^r \left( \frac{|b_{\vec{n},j} -  b_{\vec{n}-\vec{e}_\ell,j}|}{\delta^2} + 
     \frac{|a_{\vec{n},j} - a_{\vec{n}-\vec{e}_\ell,j}|}{\delta^3} \right).  \]
Observe that for every $\delta>0$ we have that $A_{\vec{n}} \to 0$ as $n \to \infty$, where $n_j = \lfloor n q_j \rfloor$. 
Hence if we choose the compact set $K^*$ such that $\delta$ is large enough, then $\hat{a}_j/\delta^2$ can be made sufficiently small
so that Lemma \ref{lim} can be applied, from which we find that $D_{\vec{n},\ell} \to 0$ uniformly for $x \in K^*$.
Note that $D_{\vec{n},\ell} \leq 2\delta$ on any compact $K \subset \mathbb{C} \setminus \mathbb{R}$, hence by Vitali's theorem
\cite[Thm. 12.8d on p.~566]{Henrici} we may conclude that $D_{\vec{n},\ell}$ converges uniformly to $0$ on every compact set
$K \subset \mathbb{C} \setminus \mathbb{R}$. This is what needed to be proved.
\end{proof}

\section{Proof of Theorem \ref{thm2}}
The proof of Theorem \ref{thm2} is very similar to the proof of Theorem \ref{thm1}. We only give the modifications which are needed.
Use the recurrence relation \eqref{rrel} but with $x$ replaced by $n^\gamma x$, and divide by $P_{\vec{n}}(n^\gamma x)$ to find
\begin{equation}  \label{ratiorec2}
    x = \frac{P_{\vec{n}+\vec{e}_k}(n^\gamma x)}{n^\gamma P_{\vec{n}}(n^\gamma x)} + \frac{b_{\vec{n},k}}{n^\gamma} 
  + \sum_{j=1}^r \frac{a_{\vec{n},j}}{n^{2\gamma}} \frac{n^\gamma P_{\vec{n}-\vec{e}_j}(n^\gamma x)}{P_{\vec{n}}(n^\gamma x)}.
\end{equation}
The partial fractions decomposition now is
\[   \frac{n^\gamma P_{\vec{n}-\vec{e}_j}(n^\gamma x)}{P_{\vec{n}}(n^\gamma x)} 
 = \sum_{k=1}^{|\vec{n}|} \frac{A_{\vec{n},k}}{x-x_{\vec{n},k}/n^\gamma}, \]
and we have the bound
\begin{equation}  \label{Pbound2}
  \left|  \frac{n^\gamma P_{\vec{n}-\vec{e}_j}(n^\gamma x)}{P_{\vec{n}}(n^\gamma x)}  \right| 
  \leq \sum_{k=1}^{|\vec{n}|} \frac{A_{\vec{n},k}}{|x-x_{\vec{n},k}/n^\gamma|} \leq \frac{1}{\delta}, 
\end{equation}
where $\delta = \min \{|x-y| : \ x \in K,\ y \in \mathbb{R} \} > 0$ is the minimal distance between $K$ and $\mathbb{R}$.
Montel's theorem gives a subsequence $(n_k)_{k \in \mathbb{N}}$ such that
\[     \lim_{k \to \infty} \frac{n_k^\gamma P_{\vec{n}_k-\vec{e}_j}(n_k^\gamma x)}{P_{\vec{n}_k}(n_k^\gamma x)} = h_j(x), \]
uniformly on $K$ for every $j$ with $1 \leq j \leq r$, where $h_j$ is analytic on $K$. From then on the proof
of Theorem \ref{thm1} can be repeated, except that one uses \eqref{ratiorec2} and \eqref{limabN} instead of \eqref{ratiorec} and \eqref{limab}. 
One finally ends up with the algebraic equation \eqref{algeq2} which is in fact the same equation as in \eqref{algeq}.

\section{Examples}  \label{sec:ex}
\subsection{Jacobi-Pi\~neiro polynomials}
These are multiple orthogonal for the Jacobi weights $d\mu_j(x) = x^{\alpha_j}(1-x)^\beta \, dx$ on $[0,1]$, where $\alpha_j, \beta >-1$
and $\alpha_i -\alpha_j \notin \mathbb{Z}$. The recurrence coefficients are given by 
\begin{multline*}   
   a_{\vec{n},j} = \frac{n_j(n_j+\alpha_j)(|\vec{n}|+\beta)}{(|\vec{n}|+n_j+\alpha_j+\beta+1)(|\vec{n}|+n_j+\alpha_j+\beta)
   (|\vec{n}|+n_j+\alpha_j+\beta-1)} \\
    \times \prod_{i=1}^r \frac{|\vec{n}|+\alpha_i+\beta}{|\vec{n}|+n_i+\alpha_i+\beta} \prod_{i\neq j} \frac{n_j+\alpha_j-\alpha_i}{n_j-n_i+\alpha_j-\alpha_i}, 
\end{multline*}
and a more complicated expression for $b_{\vec{n},j}$ (see \cite[\S 5.5]{WVA2}). We therefore have
\[  \lim_{n \to \infty} a_{\vec{n},j} = \frac{q_j^{r+1}}{(1+q_j)^3} \prod_{i=1}^r \frac{1}{1+q_i} \prod_{i\neq j} \frac{1}{q_j-q_i}.  \]
This means that the nearest neighbor recurrence coefficients $a_{\vec{n},j}$ converge only when $q_i\neq q_j$ for every $i \neq j$.
Note that the $a_{\vec{n},j}$ are not necessarily positive, but it is known that the zeros of $P_{\vec{n}}$ and $P_{\vec{n}+\vec{e}_k}$ interlace
since the measures $(\mu_1,\ldots,\mu_r)$ form an AT system (see \cite[Thm.~2.1]{HVA}).   

\subsection{Multiple Hermite polynomials}
They satisfy
\[   \int_{-\infty}^\infty H_{\vec{n}}(x) x^k e^{-x^2+c_jx}\, dx = 0, \qquad k=0,1,\ldots,n_j-1,\ 1 \leq j \leq r,  \]
where $c_j \in \mathbb{R}$ and $c_i \neq c_j$ whenever $i \neq j$. The nearest neighbor recurrence relations are
\[  xH_{\vec{n}}(x) = H_{\vec{n}+\vec{e}_k}(x) + \frac{c_k}{2} H_{\vec{n}}(x) + \sum_{j=1}^r \frac{n_j}{2} H_{\vec{n}-\vec{e}_j}(x), \]
\cite[\S 23.5]{Ismail} so that $a_{\vec{n},j} = n_j/2$ and $b_{\vec{n},k} = c_k/2$. This means that we have to use the scaling with $\gamma=1/2$ to find
\[   \lim_{n \to \infty} \frac{a_{\vec{n},j}}{n} = q_j, \quad \lim_{n \to \infty} \frac{b_{\vec{n},j}}{\sqrt{n}} = 0. \]
This is not very convenient since this implies that the limiting values $b_1,\ldots,b_r$ all coincide and $B_r$ hence has a zero
of multiplicity $r$. We can't use the partial fraction decomposition \eqref{parfrac2} and hence we cannot determine the polynomial $A_{r-1}$.
It is more interesting to let the parameters depend on $n$ and to consider $b_{\vec{n},j} = c_j \sqrt{n}$, so that
\[  \lim_{n \to \infty} \frac{b_{\vec{n},j}}{\sqrt{n}} = c_j/2 .  \]
This happens to be the case of interest when one is dealing with random matrices with external source \cite{ABKII,BK,BKI,BKIII}
and non-intersecting Brownian motions leaving from $r$ different points and arriving at one point \cite{DK,DKV}.

\subsection{Multiple Laguerre polynomials}
There are two kinds of multiple Laguerre polynomials \cite[\S 23.4]{Ismail}. The polynomials of the first kind satisfy
\[    \int_0^\infty L_{\vec{n}}(x) x^k x^{\alpha_j} e^{-x}\, dx = 0, \qquad k=0,1,\ldots, n_j-1,\ 1 \leq j \leq r, \]
where $\alpha_1,\ldots,\alpha_r >-1$ and $\alpha_i - \alpha_j \notin \mathbb{Z}$,
and the recurrence coefficients are \cite[\S 5.3]{WVA2}
\[   a_{\vec{n},j} = n_j(n_j+\alpha_j) \prod_{i\neq j} \frac{n_j+\alpha_j-\alpha_i}{n_j-n_i+\alpha_j-\alpha_i}, \quad
     b_{\vec{n},j} = |\vec{n}| + n_j + \alpha_j + 1.  \]
We can use Theorem \ref{thm2} with the scaling $\gamma=1$ and find
\[   \lim_{n \to \infty} \frac{a_{\vec{n},j}}{n^2} = q_j^{r+1} \prod_{i \neq j} \frac{1}{q_j-q_i}, \quad
     \lim_{n \to \infty} \frac{b_{\vec{n},j}}{n} = 1+q_j. \]
Observe that the limit for $a_{\vec{n},j}/n^2$ does not exist when $q_j=q_i$ for some $i$, and in that case the limit
of $b_{\vec{n},j}/n$ and $b_{\vec{n},i}/n$ is the same, so that the partial fraction decomposition \eqref{parfrac2} is not possible.
So our theorem can only be used when $q_i \neq q_j$ whenever $i\neq j$.

Multiple Laguerre polynomials of the second kind satisfy
\[  \int_0^\infty L_{\vec{n}}(x) x^k x^{\alpha} e^{-c_jx}\, dx = 0, \qquad k=0,1,\ldots, n_j-1,\ 1 \leq j \leq r, \]
where $\alpha > -1$ and $c_1,\ldots,c_r >0$ with $c_i \neq c_j$ whenever $i \neq j$. The recurrence coefficients are \cite[\S 5.3]{WVA2}
\[   a_{\vec{n},j} = \frac{n_j}{c_j^2} (|\vec{n}| + \alpha), \quad b_{\vec{n},j} = \frac{|\vec{n}|+\alpha+1}{c_j} + \sum_{i=1}^r \frac{n_i}{c_i}, \]
so that again we need the scaling $\gamma=1$. Observe that all the $a_{\vec{n},j}$ are now strictly positive whenever $n_j >0$
so that we have interlacing of the zeros of neighboring polynomials. Furthermore, we have
\[ \lim_{n \to \infty} \frac{a_{\vec{n},j}}{n^2} = \frac{q_j}{c_j^2}, \quad \lim_{n \to \infty} \frac{b_{\vec{n},j}}{n} = \frac{1}{c_j}
     + \sum_{i=1}^r \frac{q_i}{c_i}, \]
so that all the limits $b_1,\ldots,b_r$ are different. 

\subsection{Multiple Charlier polynomials}
These are discrete multiple orthogonal polynomials satisfying
\[    \sum_{k=0}^\infty C_{\vec{n}}(k) k^\ell \frac{a_j^k}{k!} = 0, \qquad \ell=0,1,\ldots,n_j-1, \ 1 \leq j \leq r, \]
where $a_1,\ldots,a_r >0$ and $a_i \neq a_j$ whenever $i \neq j$ \cite[\S 23.6.1]{Ismail}. The nearest neighbor recurrence relations are
\[ xC_{\vec{n}}(x) = C_{\vec{n}+\vec{e}_k}(x) + (a_k + |\vec{n}|) C_{\vec{n}}(x) + \sum_{j=1}^r a_jn_j C_{\vec{n}-\vec{e}_j}(x), \]
so that $a_{\vec{n},j} = a_jn_j$ and $b_{\vec{n},j} = a_j+|\vec{n}|$. Again the $a_{\vec{n},j}$ are positive whenever $n_j >0$.
We need to use a scaling with $\gamma=1$ and to scale
$a_k$ to $a_kn$ to get different limits $b_1,\ldots,b_r$. This case was worked out earlier in \cite{Francois}.

\section{Concluding remarks}

Why is it useful to investigate the ratio asymptotics for multiple orthogonal polynomials? Naturally, it is an extension of the analysis
of ratio asymptotic of orthogonal polynomials on the real line, which is connected with Nevai's class $M(a,b)$ and Rakhmanov's theorem.
Apart from that, there are three more explicit reasons 
\begin{enumerate}
  \item The ratio asymptotics allows to find the asymptotic zero distribution of the multiple orthogonal polynomials. Indeed, if $x_{\vec{n},i}$
  $(1 \leq i \leq |\vec{n}|)$ are the zeros of $P_{\vec{n}}$, then
\[   \lim_{n \to \infty} \frac{1}{|\vec{n}|} \sum_{i=1}^{|\vec{n}|} f(x_{\vec{n},i}) = \int f(t)\, d\nu(t), \]
for every bounded and continuous function $f$ on $\mathbb{R}$,
where the Stieltjes transform of the measure $\nu$ is given by
\[   \int \frac{d\nu(t)}{x-t} = \lim_{n \to \infty} \frac{1}{|\vec{n}|} \frac{P_{\vec{n}}'(x)}{P_{\vec{n}}(x)}. \]
A slight modification is needed when one uses the a scaling $n^\gamma$ as in Theorem \ref{thm2}. The asymptotic zero distribution can therefore
be obtained from the ratio asymptotics through the formula
\[  \frac{P_{\vec{n}}'(x)}{P_{\vec{n}}(x)} = \frac{P_{\vec{n}-n_r\vec{e}_j}'(x)}{P_{\vec{n}-n_r\vec{e}_j}(x)}
    + \sum_{k=0}^{n_r-1} \left( \frac{P_{\vec{n}-k\vec{e}_r}(x)}{P_{\vec{n}-(k+1)\vec{e}_r}(x)} \right)' \Big/
       \left( \frac{P_{\vec{n}-k\vec{e}_r}(x)}{P_{\vec{n}-(k+1)\vec{e}_r}(x)} \right), \]
which expresses the ratio $P_{\vec{n}}'(x)/P_{\vec{n}}(x)$ for $r$ measures as a similar ratio for $r-1$ measures
and the ratio of neighboring multiple orthogonal polynomials, for which the asymptotic behavior is given in Theorems \ref{thm1} and \ref{thm2}.
  \item It is known that there usually is a vector equilibrium problem of the form: minimize the energy functional
\[   E(\nu_1,\ldots,\nu_r) = \sum_{i=1}^r \sum_{j=1}^r C_{i,j} I(\nu_i,\nu_j) + \sum_{j=1}^r \int V_j(x)\, d\nu_j(x), \]
where $I(\nu_i,\nu_j)$ is the mutual logarithmic energy of $\nu_i$ and $\nu_j$ and $V_j$ is an external field, 
possibly with constraints $\nu_i \leq \sigma_i$ for some $1 \leq i \leq r$. To find the solution, one often needs
a Riemann surface and certain rational functions on that Riemann surface. The Riemann surface for the algebraic function $z$ satisfying
\eqref{algeq} is the natural geometric object in this case. This Riemann surface has $r+1$ sheets, it has genus 0 and there are branch points
when \eqref{algeq} has multiple roots, i.e., when $B_r^2(z)-A_{r-1}(z)B_r'(z) + A_{r-1}'(z)B_r(z) = 0$, which shows that there are $2r$ branch points.
    \item  If one wants to investigate the (strong) asymptotic behavior of the multiple orthogonal polynomials, then a useful method
    is to use the Riemann-Hilbert problem for the multiple orthogonal polynomials \cite{VAGeKu}. If one wants to use the Deift-Zhou steepest descent
    analysis for Riemann-Hilbert problems, then the first step is to transform the Riemann-Hilbert problem to a Riemann-Hilbert problem which is normalized at infinity. This requires knowledge of the (expected) asymptotic behavior of $P_{\vec{n}}(x)$, which can be expressed in terms
of the measures $(\nu_1,\ldots,\nu_r)$ from the equilibrium problem. Again the Riemann surface for \eqref{algeq} is the natural object for the Riemann-Hilbert analysis.
\end{enumerate}  

Our Theorems \ref{thm1} and \ref{thm2} have to be modified in case some of the limits $b_1,\ldots,b_r$ are equal, which is the case for
multiple Laguerre polynomials of the first kind and Jacobi-Pi\~neiro polynomials when $q_i=q_j$ for some $i \neq j$. If $b_1^*,\ldots,b_s^*$
are the distinct limits, then one has to use the polynomial $B_s(z) = (z-b_1^*)\cdots(z-b_s^*)$ of degree $s <r$. Then \eqref{defz}
implies that
\[  x-z = \sum_{j=1}^s \frac{a_j^*}{z-b_j^*}, \qquad a_j^* = \sum_{b_i=b_j^*} a_i.  \]
So one really needs the asymptotic behavior of the sum $\sum_{b_i=b_j^*} a_{\vec{n},i}$
rather than the behavior of every individual $a_{\vec{n},j}$ in this case.

\bibliographystyle{amsplain}

\end{document}